\documentclass[12pt,reqno]{article}
\usepackage{amssymb,amscd,amsmath,amsthm,color}
\newtheorem{thm}{Theorem}[section]
\newtheorem{prop}[thm]{Proposition}
\newtheorem{lemma}[thm]{Lemma}
\newtheorem{cor}[thm]{Corollary}

\newtheorem{remark}[thm]{Remark}

\def\bN{\mathbb{N}}
\def\bM{\mathbb{M}}
\def\cH{\mathcal{H}}
\def\bR{\mathbb{R}}
\def\bC{\mathbb{C}}
\def\eps{\varepsilon}
\def\<{\langle}
\def\>{\rangle}

\begin{document}
\baselineskip=16pt

\ \vskip 1cm
\centerline{\LARGE Loewner matrices of}
\bigskip
\centerline{\LARGE matrix convex and monotone functions}

\bigskip
\bigskip
\centerline{\large
Fumio Hiai\footnote{E-mail: hiai@math.is.tohoku.ac.jp}
and Takashi Sano\footnote{E-mail: sano@sci.kj.yamagata-u.ac.jp}}
	
\medskip
\begin{center}
$^{1}$\,Graduate School of Information Sciences, Tohoku University, \\
Aoba-ku, Sendai 980-8579, Japan
\end{center}
\begin{center}
$^{2}$\,Department of Mathematical Sciences, Faculty of Science, \\
Yamagata University, Yamagata 990-8560, Japan
\end{center}

\medskip
\begin{abstract}
The matrix convexity and the matrix monotony of a real $C^1$ function $f$ on $(0,\infty)$
are characterized in terms of the conditional negative or positive definiteness of the
Loewner matrices associated with $f$, $tf(t)$, and $t^2f(t)$. Similar characterizations
are also obtained for matrix monotone functions on a finite interval $(a,b)$.



\end{abstract}

\section*{Introduction}

In matrix/operator analysis quite important are the notions of matrix/operator monotone
and convex functions initiated in 1930's by L\"owner \cite{Lo} and Kraus \cite{Kr}. For a
real $C^1$ function on an interval $(a,b)$ it was proved in \cite{Lo} that $f$ is matrix
monotone of order $n$ (i.e., $A\le B$ implies $f(A)\le f(B)$ for $n\times n$ Hermitian
matrices $A,B$ with eigenvalues in $(a,b)$) if and only if the matrix
$$
L_f(t_1,\dots,t_n):=\biggl[{f(t_i)-f(t_j)\over t_i-t_j}\biggr]_{i,j=1}^n
$$
of divided differences of $f$ is positive semidefinite for any choice of $t_1,\dots,t_n$
from $(a,b)$. The above matrix $L_f(t_1,\dots,t_n)$ is called the Pick
matrix or else the Loewner (= L\"owner) matrix associated with $f$. The characterization
of matrix convex functions of similar kind was obtained in \cite{Kr} in terms of divided
differences of the second order. Almost a half century later in 1982 a modern treatment of
operator (but not matrix) convex functions was developed by Hansen and Pedersen \cite{HP}.
The most readable exposition on the subject is found in \cite{Bh}.

Recently in \cite{BS} Bhatia and the second-named author of this paper presented new
characterizations for operator convexity of nonnegative functions on $[0,\infty)$ in
terms of the conditional negative or positive definiteness (whose definitions are in
Section 1) of the Loewner matrices. More precisely, the main results in \cite{BS} are
stated as follows: A nonnegative $C^2$ function $f$ on $[0,\infty)$ with $f(0)=f'(0)=0$ is
operator convex if and only if $L_f(t_1,\dots,t_n)$ is conditionally negative definite for
all $t_1,\dots,t_n>0$ of any size $n$. Moreover, if $f$ is a nonnegative $C^3$ function on
$[0,\infty)$ with $f(0)=f'(0)=f''(0)=0$, then $f(t)/t$ is operator convex if and only if
$L_f(t_1,\dots,t_n)$ is conditionally positive definite for all $t_1,\dots,t_n>0$ of any
size $n$. More recently, Uchiyama \cite{Uc} extended, by a rather different method, the
first result stated above in such a way that the assumption $f\ge0$ is removed and the
boundary condition $f(0)=f'(0)=0$ is relaxed. Here it should be noted that the conditional
positive definiteness of the Loewner matrices and the matrix/operator monotony was related
in \cite{Ho} and \cite[Chapter XV]{Do} for a real function on a general open interval
(see Remark \ref{R-2.8} for more details).

In the present paper we consider the following conditions for a $C^1$ function $f$ on
$(0,\infty)$ and for each integer $n\ge1$:
\begin{itemize}
\item[\rm(a)$_n$] $f$ is matrix convex of order $n$ on $(0,\infty)$;
\item[\rm(b)$_n$] $\liminf_{t\to\infty}f(t)/t>-\infty$ and $L_f(t_1,\dots,t_n)$ is
conditionally negative definite for all $t_1,\dots,t_n\in(0,\infty)$;
\item[\rm(c)$_n$] $\limsup_{t\searrow0}tf(t)\ge0$ and $L_{tf(t)}(t_1,\dots,t_n)$ is
conditionally positive definite for all $t_1,\dots,t_n\in(0,\infty)$.
\end{itemize}
We improve the proof in \cite{BS} without use of integral representation of operator convex
functions and prove the implications (a)$_{2n+1}$ $\Rightarrow$ (b)$_n$, (b)$_{4n+1}$
$\Rightarrow$ (a)$_n$, (a)$_{n+1}$ $\Rightarrow$ (c)$_n$, and (c)$_{2n+1}$ $\Rightarrow$
(a)$_n$. In this way, the results in \cite{BS} (also \cite{Uc}) are refined to those in
the matrix level.

The paper is organized as follows. In Section 1 we prepare several implications among a
number of conditions related to matrix monotone and convex functions, providing technical
part of the proofs of our theorems. Some essential part of those implications are from
\cite{OT}. In Section 2 we prove the above stated theorem (Theorem \ref{T-2.1})
characterizing matrix convex functions on $(0,\infty)$ in terms of the conditional
negative or positive definiteness of the Loewner matrices. Similar characterizations of
matrix monotone functions on $(0,\infty)$ are also obtained (Theorem \ref{T-2.6}). In
Section 3 our theorems are exemplified with the power functions $t^\alpha$ on $(0,\infty)$.
(An elementary treatment of the conditional positive and negative definiteness of the
Loewner matrices for those functions is found in \cite{BS2}.) Finally in Section 4, we
further obtain similar characterizations of matrix monotone functions on a finite interval
$(a,b)$ by utilizing an operator monotone bijection between $(a,b)$ and $(0,\infty)$.

\section{Definitions and lemmas}
\setcounter{equation}{0}

For $n\in\bN$ let $\bM_n$ denote the set of all $n\times n$ complex matrices. Let $f$ be a
continuous real function on an interval $J$ of the real line. It is said that $f$ is
{\it matrix monotone of order $n$}\index{matrix monotone of order $n$} ({\it $n$-monotone} for short) on $J$ if
\begin{equation}\label{F-1.1}
A\ge B\quad\mbox{implies}\quad f(A)\ge f(B)
\end{equation}
for Hermitian matrices $A,B$ in $\bM_n$ with $\sigma(A),\sigma(B)\subset J$, where
$\sigma(A)$ stands for the spectrum (the eigenvalues) of $A$. It is said that $f$ is
{\it matrix convex of order $n$} ({\it $n$-convex} for short) on $J$ if
\begin{equation}\label{F-1.2}
f(\lambda A+(1-\lambda)B)\le\lambda f(A)+(1-\lambda)f(B)
\end{equation}
for all Hermitian $A,B\in\bM_n$ with $\sigma(A),\sigma(B)\subset J$ and for all
$\lambda\in(0,1)$. Also, $f$ is said to be {\it $n$-concave} on $J$ if $-f$ is $n$-convex
on $J$. Furthermore, it is said that $f$ is {\it operator monotone} on $J$ if
\eqref{F-1.1} holds for self-adjoint operators $A,B$ in $B(\cH)$ with
$\sigma(A),\sigma(B)\subset J$, and {\it operator convex} on $J$ if \eqref{F-1.2} holds
for all self-adjoint $A,B\in B(\cH)$ with $\sigma(A),\sigma(B)\subset J$ and for all
$\lambda\in(0,1)$, where $B(\cH)$ is the set of all bounded operators on an
infinite-dimensional (separable) Hilbert space $\cH$. As is well known, $f$ is operator
monotone (resp., operator convex) on $J$ if and only if it is $n$-monotone (resp.,
$n$-convex) on $J$ for all $n\in\bN$.

For each $n\in\bN$ let $\bC_0^n$ denote the subspace of $\bC^n$ consisting of all
$x=(x_1,\dots,x_n)^t\in\bC^n$ such that $\sum_{i=1}^nx_i=0$. A Hermitian matrix $A$ in
$\bM_n$ is said to be {\it conditionally positive definite} ({\it c.p.d.}\ for short) if
$\<x,Ax\>\ge0$ for all $x\in\bC_0^n$, and {\it conditionally negative definite}
({\it c.n.d.}\ for short) if $-A$ is c.p.d. Let $f$ be a real $C^1$ (i.e., continuously
differentiable) function $f$ on an interval $(a,b)$ with $-\infty\le a<b\le\infty$. The
{\it divided difference} of $f$ is defined by
$$
f^{[1]}(s,t):=\begin{cases}
{f(s)-f(t)\over s-t} & \text{if $s\ne t$}, \\
f'(s) & \text{if $s=t$},
\end{cases}
$$
which is a continuous function on $(a,b)^2$ (see \cite[Chapter I]{Do} for details on
divided differences). For each $t_1,\dots,t_n\in(a,b)$, the {\it Loewner matrix}
$L_f(t_1,\dots,t_n)$ associated with $f$ (for $t_1,\dots,t_n$) is defined to be the
$n\times n$ matrix whose $(i,j)$-entry is $f^{[1]}(t_i,t_j)$, i.e.,
$$
L_f(t_1,\dots,t_n):=\bigl[f^{[1]}(t_i,t_j)\bigr]_{i,j=1}^n.
$$
In the fundamental paper \cite{Lo}, Karl L\"owner (later Charles Loewner) proved that,
for a real $C^1$ function $f$ on $(a,b)$ and for each $n\in\bN$, $f$ is $n$-monotone on
$(a,b)$ if and only if $L_f(t_1,\dots,t_n)$ is positive semidefinite for any choice of
$t_1,\dots,t_n$ from $(a,b)$.

Let $f$ be a continuous real function on $[0,\infty)$. For each $n\in\bN$ we consider the
following conditions:
\begin{itemize}
\item[(i)$_n$] $f$ is $n$-monotone on $[0,\infty)$;
\item[(ii)$_n$] $f$ is $n$-concave on $[0,\infty)$;
\item[(iii)$_n$] $f$ is $n$-convex on $[0,\infty)$ and $f(0)\le0$;
\item[(iv)$_n$] $f(X^*AX)\le X^*f(A)X$ for all $A,X\in\bM_n$ with $A\ge0$ and $\|X\|\le1$;
\item[(v)$_n$] $f(t)/t$ is $n$-monotone on $(0,\infty)$.
\end{itemize}

When $f$ is $C^1$ on $(0,\infty)$, we further consider the following conditions:
\begin{itemize}
\item[(vi)$_n$] $L_f(t_1,\dots,t_n)$ is c.n.d.\ for all $t_1,\dots,t_n\in(0,\infty)$;
\item[(vii)$_n$] $L_{tf(t)}(t_1,\dots,t_n)$ is c.p.d.\ for all
$t_1,\dots,t_n\in(0,\infty)$.
\end{itemize}

For a continuous real function $f$ on $[0,\infty)$ such that $f(t)>0$ for all $t>0$, the
following conditions are also considered:
\begin{itemize}
\item[(viii)$_n$] $t/f(t)$ is $n$-monotone on $(0,\infty)$;
\item[(ix)$_n$] $t^2/f(t)$ is $n$-monotone on $(0,\infty)$.
\end{itemize}

In the rest of this section we present lemmas on several relations among the above
conditions, which will be used in the next section. But they may be of some independent
interest.

\begin{lemma}\label{L-1.1}
Let $f$ be a continuous real function on $[0,\infty)$. Then for every $n\in\bN$ the
following implications hold:
$$
{\rm(iii)}_{n+1}\Longrightarrow{\rm(iv)}_n\Longleftrightarrow{\rm(v)}_n,\qquad
{\rm(v)}_{2n}\Longrightarrow{\rm(iii)}_n.
$$
\end{lemma}

\begin{proof}
(iii)$_{n+1}$ $\Rightarrow$ (v)$_n$ was shown in \cite[Theorem 2.2]{OT}, and (iv)$_n$
$\Leftrightarrow$ (v)$_n$ was in \cite[Theorem 2.1]{OT} while the following proof is
comparatively simpler. Indeed, (iv)$_n$ $\Rightarrow$ (v)$_n$ is seen from the proof of
\cite[Theorem 2.4]{HP}. Conversely, suppose (v)$_n$, and let $A\in\bM_n$ be positive
semidefinite and $X\in\bM_n$ with $\|X\|\le1$. We may assume that $A>0$, and we further
assume that $X$ is invertible. Take the polar decomposition $A^{1/2}X=U|A^{1/2}X|$ and set
$B:=|X^*A^{1/2}|^2$. Then we have $B\le A$ and $B^{1/2}=U|A^{1/2}X|U^*=A^{1/2}XU^*$, so
$A^{-1/2}B^{1/2}=XU^*$. Since $B^{-1/2}f(B)B^{-1/2}\le A^{-1/2}f(A)A^{-1/2}$, we have
$$
f(B)\le B^{1/2}A^{-1/2}f(A)A^{-1/2}B^{1/2}=UX^*f(A)XU^*
$$
and $f(B)=Uf(X^*AX)U^*$. Therefore, $f(X^*AX)\le X^*f(A)X$. When $X$ is not invertible,
choose a sequence $\eps_k\to0$ such that $X_k:=(1+|\eps_k|)^{-1}(X+\eps_kI)$ is invertible
for any $k$, and take the limit of $f(X_k^*AX_k)\le X_k^*f(A)X_k$. The remaining
(v)$_{2n}$ $\Rightarrow$ (iii)$_n$ is seen from the proof of
\cite[Theorems 2.1 and 2.4]{HP}.
\end{proof}

\begin{lemma}\label{L-1.2}
Let $f$ be a continuous real function on $[0,\infty)$. Then for every $n\in\bN$ the
implication
$$
{\rm(i)}_{2n}\Longrightarrow{\rm(ii)}_n
$$
holds. Moreover, if $f(t)>0$ for all $t>0$, then for every $n\in\bN$ the following hold:
$$
{\rm(ii)}_n\Longrightarrow{\rm(i)}_n,\qquad
{\rm(i)}_{2n}\Longrightarrow{\rm(viii)}_n.
$$
\end{lemma}

\begin{proof}
(i)$_{2n}$ $\Rightarrow$ (ii)$_n$ is seen from the proof of \cite[Theorem 2.4]{Uc}. Now
assume that $f(t)>0$ for all $t>0$. Then (ii)$_n$ $\Rightarrow$ (i)$_n$ is seen from the
proof of \cite[Theorem 2.5]{HP}. Next, suppose (i)$_{2n}$. Since $f$ is $2n$-monotone on
$[0,\infty)$ with $-f\le0$, the proof of \cite[Theorem 2.5]{HP} shows that $-f$ satisfies
(iv)$_n$ and hence (v)$_n$ by Lemma \ref{L-1.1}, so $-f(t)/t$ is $n$-monotone on
$(0,\infty)$. Since $-t^{-1}$ is operator monotone on $(-\infty,0)$, it follows that
$t/f(t)=-(-f(t)/t)^{-1}$ is $n$-monotone on $(0,\infty)$. Hence (viii)$_n$ follows.
\end{proof}

Let $f$ be as in Lemma \ref{L-1.2} such that $f(t)>0$ for all $t>0$. Since (viii)$_n$ is
equivalent to the $n$-monotony of $-f(t)/t$ on $(0,\infty)$, we further have
(viii)$_{2n}$ $\Rightarrow$ (ii)$_n$ and (ii)$_{n+1}$ $\Rightarrow$ (viii)$_n$ by applying
Lemma \ref{L-1.1} to $-f$, though not used in the rest of the paper.

\begin{lemma}\label{L-1.3}
Let $f$ be a continuous real function on $[0,\infty)$ such that $f(t)>0$ for all $t>0$.
Then for every $n\in\bN$ the following hold:
$$
{\rm(v)}_{2n}\Longrightarrow{\rm(ix)}_n,\qquad
{\rm(ix)}_{2n}\Longrightarrow{\rm(v)}_n.
$$
\end{lemma}

\begin{proof}
Since $t^2/f(t)=t/(f(t)/t)$ and $f(t)/t=t/(t^2/f(t))$, the stated implications are
immediately seen from (i)$_{2n}$ $\Rightarrow$ (viii)$_n$ of Lemma \ref{L-1.2}.
\end{proof}

\begin{lemma}\label{L-1.4}
Let $f$ be a real $C^1$ function on $[0,\infty)$ such that $f(t)>0$ for all $t>0$,
$f(0)=0$, and $f'(0)\ge0$. Then for every $n\in\bN$ the following implications hold:
$$
{\rm(vi)}_{n+1}\Longrightarrow{\rm(ix)}_n\Longrightarrow{\rm(vi)}_n.
$$
\end{lemma}

\begin{proof}
(vi)$_{n+1}$ $\Rightarrow$ (ix)$_n$.\enspace
First, recall (see \cite[p.\ 193]{BR} or \cite[p.\ 134]{Do}) that if a Hermitian
$(n+1)\times(n+1)$ matrix $[a_{ij}]_{i,j=1}^{n+1}$ is c.p.d., then the $n\times n$ matrix
$$
\bigl[a_{ij}-a_{i,n+1}-a_{n+1,j}+a_{n+1,n+1}\bigr]_{i,j=1}^n
$$
is positive semidefinite. Hence for every $t_1,\dots,t_n,t_{n+1}\in(0,\infty)$,
assumption (vi)$_{n+1}$ implies that
$$
\bigl[f^{[1]}(t_i,t_j)-f^{[1]}(t_i,t_{n+1})-f^{[1]}(t_j,t_{n+1})
+f'(t_{n+1})\bigr]_{i,j=1}^n\le0.
$$
Since $f(0)=0$, letting $t_{n+1}\searrow0$ yields that
$$
\biggl[f^{[1]}(t_i,t_j)-{f(t_i)\over t_i}-{f(t_j)\over t_j}+f'(0)\biggr]_{i,j=1}^n\le0.
$$
Since
\begin{equation}\label{F-1.3}
f^{[1]}(t_i,t_j)-{f(t_i)\over t_i}-{f(t_j)\over t_j}
=-{f(t_i)\over t_i}\cdot\biggl({t^2\over f(t)}\biggr)^{[1]}(t_i,t_j)
\cdot{f(t_j)\over t_j},
\end{equation}
we see that
$$
\biggl[{f(t_i)\over t_i}\cdot\biggl({t^2\over f(t)}\biggr)^{[1]}(t_i,t_j)
\cdot{f(t_j)\over t_j}\biggr]_{i,j=1}^n-f'(0)E_n\ge0,
$$
where $E_n$ stands for the $n\times n$ matrix of all entries equal to $1$. Since
$f'(0)\ge0$, we have $L_{t^2/f(t)}(t_1,\dots,t_n)\ge0$, which yields (ix)$_n$ by
L\"owner's theorem.

(ix)$_n$ $\Rightarrow$ (vi)$_n$.\enspace
For every $t_1,\dots,t_n\in(0,\infty)$, it follows from \eqref{F-1.3} that
$$
L_f(t_1,\dots,t_n)
=-\biggl[{f(t_i)\over t_i}\cdot\biggl({t^2\over f(t)}\biggr)^{[1]}(t_i,t_j)
\cdot{f(t_j)\over t_j}\biggr]_{i,j=1}^n
+\biggl[{f(t_i)\over t_i}+{f(t_j)\over t_j}\biggr]_{i,j=1}^n.
$$
Since $L_{t^2/f(t)}(t_1,\dots,t_n)\ge0$ by assumption (ix)$_n$, the above expression yields
that $L_f(t_1,\dots,t_n)$ is c.n.d.
\end{proof}

The proof of the next lemma is a modification of the argument in \cite[p.\ 428]{Ho}.

\begin{lemma}\label{L-1.5}
Let $f$ be a continuous real function on $[0,\infty)$ with $f(0)=0$ such that $f$ is
$C^1$ on $(0,\infty)$ and $\lim_{t\searrow0}tf'(t)=0$. {\rm(}This is the case if $f$ is
$C^1$ on $[0,\infty)$ with $f(0)=0$.{\rm)} Then for every $n\in\bN$ the following
implications hold:
$$
{\rm(vii)}_{n+1}\Longrightarrow{\rm(v)}_n\Longrightarrow{\rm(vii)}_n.
$$
\end{lemma}

\begin{proof}
(vii)$_{n+1}$ $\Rightarrow$ (v)$_n$.\enspace
Set $g(t):=tf(t)$ for $t\in[0,\infty)$ and for each $\eps>0$ define
$$
g_\eps(t):=g(t+\eps)-g(\eps)-g'(\eps)t,\qquad t\in[0,\infty).
$$
Then $g_\eps$ is $C^1$ on $[0,\infty)$ and $g_\eps(0)=g_\eps'(0)=0$. From assumption
(vii)$_{n+1}$ it follows that $L_{g_\eps}(t_1,\dots,t_n,t_{n+1})$ is c.p.d.\ for every
$t_1,\dots,t_n,t_{n+1}\in(0,\infty)$. Hence similarly to the proof of Lemma \ref{L-1.4}
we have
$$
\biggl[g_\eps^{[1]}(t_i,t_j)
-{g_\eps(t_i)\over t_i}-{g_\eps(t_j)\over t_j}\biggr]_{i,j=1}^n\ge0
$$
for every $t_1,\dots,t_n\in(0,\infty)$. Since
\begin{equation}\label{F-1.4}
g_\eps^{[1]}(t_i,t_j)-{g_\eps(t_i)\over t_i}-{g_\eps(t_j)\over t_j}
=t_i\cdot\biggl({g_\eps(t)\over t^2}\biggr)^{[1]}(t_i,t_j)\cdot t_j,
\end{equation}
we see that $L_{g_\eps(t)/t^2}(t_1,\dots,t_n)\ge0$. Since $g(\eps)\to0$ and
$g'(\eps)=f(\eps)+\eps f'(\eps)\to0$ as $\eps\searrow0$ thanks to assumption on $f$, it
follows that $g_\eps(t)/t^2\to g(t)/t^2=f(t)/t$ as $\eps\searrow0$ for any $t>0$. Hence
we have $L_{f(t)/t}(t_1,\dots,t_n)\ge0$, which yields (v)$_n$.

(v)$_n$ $\Rightarrow$ (vii)$_n$.\enspace
Let $g$ be as above. For every $t_1,\dots,t_n\in(0,\infty)$, from \eqref{F-1.4} for $g$
instead of $g_\eps$ we have
$$
L_g(t_1,\dots,t_n)
=\biggl[t_i\cdot\biggl({f(t)\over t}\biggr)^{[1]}(t_i,t_j)\cdot t_j\biggr]_{i,j=1}^n
+\biggl[{g(t_i)\over t_i}+{g(t_j)\over t_j}\biggr]_{i,j=1}^n,
$$
which is c.p.d.\ due to (v)$_n$.
\end{proof}

\section{Functions on $(0,\infty)$}
\setcounter{equation}{0}

The aim of this section is to relate the $n$-convexity and the $n$-monotony of a $C^1$
function on $(0,\infty)$ to the c.p.d.\ and the c.n.d.\ of the Loewner matrices associated
with certain corresponding functions. The first theorem is concerned with $n$-convex
functions on $(0,\infty)$.

\begin{thm}\label{T-2.1}
Let $f$ be a real $C^1$ function on $(0,\infty)$. For each $n\in\bN$ consider the
following conditions:
\begin{itemize}
\item[\rm(a)$_n$] $f$ is $n$-convex on $(0,\infty)$;
\item[\rm(b)$_n$] $\liminf_{t\to\infty}f(t)/t>-\infty$ and $L_f(t_1,\dots,t_n)$ is
c.n.d.\ for all $t_1,\dots,t_n\in(0,\infty)$;
\item[\rm(c)$_n$] $\limsup_{t\searrow0}tf(t)\ge0$ and $L_{tf(t)}(t_1,\dots,t_n)$ is
c.p.d.\ for all $t_1,\dots,t_n\in(0,\infty)$.
\end{itemize}
Then for every $n\in\bN$ the following implications hold:
$$
{\rm(a)}_{2n+1}\Longrightarrow{\rm(b)}_n,\quad
{\rm(b)}_{4n+1}\Longrightarrow{\rm(a)}_n,\quad
{\rm(a)}_{n+1}\Longrightarrow{\rm(c)}_n,\quad
{\rm(c)}_{2n+1}\Longrightarrow{\rm(a)}_n.
$$
\end{thm}

\begin{proof}
First, note that $\lim_{t\to\infty}f(t)/t>-\infty$ (the limit may be $+\infty$) and
$\liminf_{t\searrow0}tf(t)\allowbreak\ge0$, slightly stronger than the boundary conditions
in (b)$_n$ and (c)$_n$, are satisfied as long as $f$ satisfies (a)$_1$, i.e., $f$ is
convex as a numerical function on $(0,\infty)$. When $\liminf_{t\to\infty}f(t)/t>-\infty$,
for any $\eps>0$ it follows that
$$
\inf_{t\in(0,\infty)}{f(t+\eps)-f(\eps)\over t}>-\infty.
$$
So one can choose a $\gamma_\eps\in\bR$ smaller than the above infimum and define
$$
f_\eps(t):=f(t+\eps)-f(\eps)-\gamma_\eps t,\qquad t\in[0,\infty),
$$
so that $f_\eps(t)>0$ for all $t\in(0,\infty)$, $f_\eps(0)=0$ and $f_\eps'(0)>0$. In the
proof below, $f_\eps$ will be such a function chosen for each $\eps>0$.

(a)$_{2n+1}$ $\Rightarrow$ (b)$_n$.\enspace
For any $\eps>0$, since (a)$_{2n+1}$ implies that $f_\eps$ is $(2n+1)$-convex on
$[0,\infty)$, one can apply (iii)$_{2n+1}$ $\Rightarrow$ (v)$_{2n}$ $\Rightarrow$ (ix)$_n$
$\Rightarrow$ (vi)$_n$ of Lemmas \ref{L-1.1}, \ref{L-1.3}, and \ref{L-1.4} to $f_\eps$ so
that $L_{f_\eps}(t_1,\dots,t_n)$ is c.n.d.\ for every $t_1,\dots,t_n\in(0,\infty)$. Since
\begin{equation}\label{F-2.1}
L_{f_\eps}(t_1,\dots,t_n)=L_f(t_1+\eps,\dots,t_n+\eps)-\gamma_\eps E_n,
\end{equation}
it follows that $L_f(t_1+\eps,\dots,t_n+\eps)$ is c.n.d. Hence (b)$_n$ holds since
$\eps>0$ is arbitrary.

(b)$_{4n+1}$ $\Rightarrow$ (a)$_n$.\enspace
For any $\eps>0$, thanks to \eqref{F-2.1} with $4n+1$ in place of $n$, it follows from
(b)$_{4n+1}$ that (vi)$_{4n+1}$ is satisfied for $f_\eps$. So one can apply (vi)$_{4n+1}$
$\Rightarrow$ (ix)$_{4n}$ $\Rightarrow$ (v)$_{2n}$ $\Rightarrow$ (iii)$_n$ of Lemmas
\ref{L-1.4}, \ref{L-1.3}, and \ref{L-1.1} to $f_\eps$ so that $f_\eps$ is $n$-convex on
$[0,\infty)$. Hence $f(t+\eps)$ is $n$-convex on $[0,\infty)$ so that (a)$_n$ follows
since $\eps>0$ is arbitrary.

(a)$_{n+1}$ $\Rightarrow$ (c)$_n$.\enspace
For any $\eps>0$, since $f_\eps$ is $(n+1)$-convex on $[0,\infty)$, we can apply
(iii)$_{n+1}$ $\Rightarrow$ (v)$_n$ $\Rightarrow$ (vii)$_n$ of Lemmas \ref{L-1.1} and
\ref{L-1.5} to $f_\eps$, so $L_{tf_\eps(t)}(t_1,\dots,t_n)$ is c.p.d.\ for every
$t_1,\dots,t_n\in(0,\infty)$. Since
$$
L_{tf_\eps(t)}(t_1,\dots,t_n)
=L_{tf(t+\eps)}(t_1,\dots,t_n)-f(\eps)E_n-\gamma_\eps\bigl[t_i+t_j\bigr]_{i,j=1}^n,
$$
we see that $L_{tf(t+\eps)}(t_1,\dots,t_n)$ is c.p.d. Furthermore, since
$tf(t+\eps)\to tf(t)$ and
$$
(tf(t+\eps))'=f(t+\eps)+tf'(t+\eps)\longrightarrow f(t)+tf'(t)=(tf(t))'
$$
as $\eps\searrow0$ for any $t>0$, it follows that $L_{tf(t)}(t_1,\dots,t_n)$ is c.p.d.
Hence (c)$_n$ holds.

(c)$_{2n+1}$ $\Rightarrow$ (a)$_n$.\enspace
Let $g(t):=tf(t)$ for $t\in(0,\infty)$. Since $\limsup_{t\searrow0}g(t)\ge0$ by assumption,
one can choose a sequence $\eps_k\searrow0$ in such a way that $g(\eps_k)>0$ for all $k$
when $\limsup_{t\searrow0}g(t)>0$, or else $\lim_{k\to\infty}g(\eps_k)=0$ when
$\limsup_{t\searrow0}g(t)=0$. Define
$$
g_k(t):=g(t+\eps_k)-g(\eps_k)-g'(\eps_k)t,\qquad t\in[0,\infty).
$$
Thanks to $\lim_{t\searrow0}g_k(t)/t=0$, $g_k$ is written as $g_k(t)=tf_k(t)$ with a
continuous function $f_k$ on $[0,\infty)$ with $f_k(0)=0$. Notice that $f_k$ is obviously
$C^1$ on $(0,\infty)$ and furthermore
\begin{align*}
tf_k'(t)&=g_k'(t)-{g_k(t)\over t} \\
&=(g'(t+\eps_k)-g'(\eps_k))-\biggl({g(t+\eps_k)-g(\eps_k)\over t}-g'(\eps_k)\biggr) \\
&\longrightarrow0\quad\mbox{as $t\searrow0$}.
\end{align*}
Since (c)$_{2n+1}$ implies that (vii)$_{2n+1}$ is satisfied for $f_k$, we can apply
(vii)$_{2n+1}$ $\Rightarrow$ (v)$_{2n}$ $\Rightarrow$ (iii)$_n$ of Lemmas \ref{L-1.5} and
\ref{L-1.1} to $f_k$ so that $f_k$ is $n$-convex on $[0,\infty)$. Writing
$$
f_k(t)={(t+\eps_k)f(t+\eps_k)\over t}-{g(\eps_k)\over t}-g'(\eps_k),
\qquad t>0,
$$
we see that
$$
\tilde f_k(t):={(t+\eps_k)f(t+\eps_k)\over t}-{g(\eps_k)\over t}
$$
is $n$-convex on $(0,\infty)$. When $g(\eps_k)>0$ for all $k$,
$$
{(t+\eps_k)f(t+\eps_k)\over t}=\tilde f_k(t)+{g(\eps_k)\over t}
$$
is $n$-convex on $(0,\infty)$ since $g(\eps_k)/t$ is operator convex on $(0,\infty)$.
Furthermore, notice that $\lim_{k\to\infty}(t+\eps_k)f(t+\eps_k)/t=f(t)$ for all $t>0$.
Hence (a)$_n$ holds. On the other hand, when $\lim_{k\to\infty}g(\eps_k)=0$, we have
$\lim_{k\to\infty}\tilde f_k(t)=f(t)$ for all $t>0$, and hence (a)$_n$ holds as well.
\end{proof}

The equivalence of the following (a)--(c) immediately follows from Theorem \ref{T-2.1},
which extends \cite[Theorems 1.1, 1.2, 1.4, and 1.5]{BS}. The equivalence between
(a) and (b) was proved in \cite[Theorem 3.1]{Uc} by a different method.

\begin{cor}\label{C-2.2}
Let $f$ be a real $C^1$ function on $(0,\infty)$. Then the following conditions are
equivalent:
\begin{itemize}
\item[\rm(a)] $f$ is operator convex on $(0,\infty)$;
\item[\rm(b)] $\liminf_{t\to\infty}f(t)/t>-\infty$ and $L_f(t_1,\dots,t_n)$ is c.n.d.\ for
all $n\in\bN$ and all $t_1,\dots,t_n\allowbreak\in(0,\infty)$;
\item[\rm(c)] $\limsup_{t\searrow0}tf(t)\ge0$ and $L_{tf(t)}(t_1,\dots,t_n)$ is c.p.d.\ for
all $n\in\bN$ and all $t_1,\dots,t_n\in(0,\infty)$.
\end{itemize}

Moreover, if the above conditions are satisfied, then $\lim_{t\to\infty}f(t)/t$ and
$\lim_{t\searrow0}tf(t)$ exist in $[0,\infty)$ and $[0,\infty)$, respectively.
\end{cor}

\begin{proof}
It remains to show the last assertion. Assume that $f$ is operator convex on $(0,\infty)$.
Then $\lim_{t\to\infty}f(t)/t>-\infty$ is obvious as noted at the beginning of the proof
of Theorem \ref{T-2.1}. Consider the function $g(t):=f^{[1]}(t,1)$ on $(0,\infty)$. Then
the characterization of operator convex functions due to Kraus \cite{Kr} says that $g$ is
operator monotone function on $(0,\infty)$ and so $g(t+1)$ is operator monotone on
$(-1,1)$. By L\"owner's theorem \cite{Lo} (or \cite[V.4.5]{Bh}) we have the integral
representation
$$
g(t+1)=g(1)+g'(1)\int_{[-1,1]}{t\over1-\lambda t}\,d\mu(\lambda),\qquad t\in(-1,1)
$$
with a probability measure $\mu$ on $[-1,1]$. Letting $\alpha:=\mu(\{-1\})$ we write
$$
(t+1)g(t+1)=g(1)(t+1)+\alpha g'(1)t+\int_{(-1,1]}{t(t+1)\over1-\lambda t}\,d\mu(\lambda),
\qquad t\in(-1,1).
$$
Since $(t+1)/(1-\lambda t)\le1$ for all $\lambda\in(-1,1]$ and $t\in(-1,0]$, the Lebesgue
convergence theorem yields that
$$
\lim_{t\searrow0}tg(t)=\lim_{t\searrow-1}(t+1)g(t+1)=-\alpha g'(1),
$$
from which $\lim_{t\searrow0}tf(t)=\alpha g'(1)\in[0,\infty)$ immediately follows.
\end{proof}

\begin{remark}\label{R-2.3}\rm
Concerning the operator convex functions $g_\lambda(t):=t^2/(1-\lambda t)$ on $(-1,1)$
with $\lambda\in[-1,1]$ (see \cite[p.\ 134]{Bh}), it was shown in \cite[Theorem 3.1]{BS2}
that $L_{g_\lambda}(t_1,\dots,t_n)$ is c.n.d.\ if $\lambda\in[-1,0]$ and c.p.d.\ if
$\lambda\in[0,1]$ for every $t_1,\dots,t_n\in(-1,1)$ of any size $n$. By considering
$g_\lambda|_{(0,1)}$ and $-g_\lambda|_{(0,1)}$ with $\lambda\in(0,1)$, we see that
neither (a) $\Rightarrow$ (b) nor (b) $\Rightarrow$ (a) of Corollary \ref{C-2.2} can be
extended to functions on a finite open interval $(0,b)$.
\end{remark}

\begin{remark}\label{R-2.4}\rm
The conditions $\limsup_{t\to\infty}f(t)/t>-\infty$ and $\limsup_{t\searrow0}tf(t)\ge0$
are obviously satisfied if $f(t)\ge0$ for all $t>0$. We remark that these boundary
conditions are essential in Theorem \ref{T-2.1} and Corollary \ref{C-2.2}, as seen from
the following discussions.

When $1\le\alpha\le2$, the function $t^\alpha$ is operator convex on $(0,\infty)$. Hence
Corollary \ref{C-2.2} implies that $L_{t^{\alpha+1}}(t_1,\dots,t_n)$ is c.p.d.\ and so
$L_{-t^{\alpha+1}}(t_1,\dots,t_n)$ is c.n.d.\ for all $t_1,\dots,t_n\in(0,\infty)$,
$n\in\bN$. However, $-t^{\alpha+1}$ is not operator convex (even not convex as a numerical
function) on $(0,\infty)$. Note that $\lim_{t\to\infty}(-t^{\alpha+1})/t=-\infty$.

When $-1\le\alpha\le0$, the function $t^\alpha$ is operator convex on $(0,\infty)$. Hence
Corollary \ref{C-2.2} implies that $L_{t^\alpha}(t_1,\dots,t_n)$ is c.n.d.\ and so
$L_{-t^\alpha}(t_1,\dots,t_n)$ is c.p.d.\ for all $t_1,\dots,t_n\in(0,\infty)$, $n\in\bN$.
However, $-t^{\alpha-1}$ is not operator convex (even not convex as a numerical function)
on $(0,\infty)$. Note that $\lim_{t\searrow0}t(-t^{\alpha-1})\le-1$.
\end{remark}

A problem arising from Theorem \ref{T-2.1} would be to determine the minimal number
$\nu(n)$ (resp., $\pi(n)$) of $m\in\bN$ such that (b)$_m$ $\Rightarrow$ (a)$_n$
(resp., (c)$_m$ $\Rightarrow$ (a)$_n$) for all real $C^1$ functions on $(0,\infty)$. The
problem does not seem easy even for the case $n=2$ while $3\le\nu(2)\le9$ and
$3\le\pi(2)\le5$ (see Proposition \ref{P-3.1} for (b)$_2$ $\not\Rightarrow$ (a)$_2$ and
(c)$_2$ $\not\Rightarrow$ (a)$_2$). In the case $n=1$, the c.n.d.\ condition of (b)$_1$
and the c.p.d.\ condition of (c)$_1$ are void but (a)$_1$ means that $f$ is simply convex
on $(0,\infty)$. Hence the next proposition shows that $\nu(1)=\pi(1)=2$, which will be
used in the proof of the next theorem.

\begin{prop}\label{P-2.5}
Let $f$ be a real $C^1$ function on $(0,\infty)$. Then for conditions {\rm(a)}$_1$,
{\rm(b)}$_2$, and {\rm(c)}$_2$ of Theorem \ref{T-2.1} the following hold:
$$
{\rm(b)}_2\Longrightarrow{\rm(a)}_1,\qquad
{\rm(c)}_2\Longrightarrow{\rm(a)}_1.
$$
\end{prop}

\begin{proof}
(b)$_2$ $\Rightarrow$ (a)$_1$.\enspace
The c.n.d.\ condition of (b)$_2$ is equivalent to the concavity of $f'$ on $(0,\infty)$
(see \cite[p.\ 137, Lemma 3]{Do}). Now suppose that $f'$ is not non-decreasing; then
$\lim_{t\to\infty}f'(t)=-\infty$ from concavity. Hence for any $K>0$ an $a>0$ can be chosen
so that $f'(s)<-K$ for all $s>a$. For every $t>a$, since
$$
{f(t)-f(a)\over t-a}=f'(s)<-K\quad\mbox{for some $s\in(a,t)$},
$$
we have
$$
\limsup_{t\to\infty}{f(t)\over t}
=\limsup_{t\to\infty}{f(t)-f(a)\over t-a}\le-K,
$$
which implies that $\lim_{t\to\infty}f(t)/t=-\infty$, contradicting the assumption.
Hence $f'$ is non-decreasing, so $f$ is convex on $(0,\infty)$.

(c)$_2$ $\Rightarrow$ (a)$_1$.\enspace
Write $g(t):=tf(t)$ for $t\in(0,\infty)$. The c.p.d.\ condition of (c)$_2$ is equivalent
to the convexity of $g'$ on $(0,\infty)$. From this and the assumption
$\limsup_{t\searrow0}g(t)\ge0$ it follows that the limit $\lim_{t\searrow0}g(t)$ exists
and is in $[0,\infty)$. Hence we may assume that $g$ is continuous on $[0,\infty)$ with
$g(0)\ge0$. Notice that
$$
f(t)={g(t)\over t}={g(0)\over t}+{1\over t}\int_0^tg'(s)\,ds
={g(0)\over t}+\lim_{\eps\searrow0}{1\over t}\int_0^tg'(s+\eps)\,ds,
\quad t>0.
$$
Hence the conclusion follows from the fact \cite{BO} that if $h$ is a continuous convex
function on $[0,\infty)$, then the function ${1\over t}\int_0^th(s)\,ds$ is convex on
$(0,\infty)$. For the convenience of the reader a short proof is given here. Indeed, such
a function $h$ can be approximated uniformly on each finite interval $[0,a]$ by functions
of the form
$$
\alpha t+\beta+\sum_{i=1}^k\alpha_i(t-\lambda_i)_+
$$
with $\alpha,\beta\in\bR$ and $\alpha_i,\lambda_i>0$, where $x_+:=\max\{x,0\}$ for
$x\in\bR$. Since the function ${1\over t}\int_0^t(s-\lambda)_+\,ds=(t-\lambda)_+^2/2t$ is
convex on $(0,\infty)$ for any $\lambda>0$, the assertion follows.
\end{proof}

Note that the converse of each implication of Proposition \ref{P-2.5} is invalid. Indeed,
for the second consider the function
$$
f(t):=\begin{cases}
t^2, & \text{$0\le t\le1$}, \\
2t-1, & \text{$t\ge1$},
\end{cases}
$$
and the function $t^3$ for the first (see Proposition \ref{P-3.1}).

The next theorem is concerned with $n$-monotone functions on $(0,\infty)$.

\begin{thm}\label{T-2.6}
Let $f$ be a real $C^1$ function on $(0,\infty)$. For each $n\in\bN$ consider the following
conditions:
\begin{itemize}
\item[\rm(a)$_n'$] $f$ is $n$-monotone on $(0,\infty)$;
\item[\rm(b)$_n'$] $\limsup_{t\to\infty}f(t)/t<+\infty$,
$\limsup_{t\to\infty}f(t)>-\infty$, and $L_f(t_1,\dots,t_n)$ is c.p.d.\ for all
$t_1,\dots,t_n\in(0,\infty)$;
\item[\rm(c)$_n'$] $\liminf_{t\searrow0}tf(t)\le0$,
$\limsup_{t\to\infty}f(t)>-\infty$, and $L_{tf(t)}(t_1,\dots,t_n)$ is c.n.d.\ for all
$t_1,\dots,t_n\in(0,\infty)$;
\item[\rm(d)$_n'$] $\liminf_{t\searrow0}tf(t)\le0$,
$\limsup_{t\searrow0}t^2f(t)\ge0$, and $L_{t^2f(t)}(t_1,\allowbreak\dots,t_n)$ is
c.p.d.\ for all $t_1,\dots,t_n\in(0,\infty)$.
\end{itemize}
Then for every $n\in\bN$ the following implications hold:
$$
{\rm(a)}_n'\Longrightarrow{\rm(b)}_n'\ \mbox{if $n\ge2$},\quad
{\rm(b)}_{4n+1}'\Longrightarrow{\rm(a)}_n',\quad
{\rm(a)}_{2n+2}'\Longrightarrow{\rm(c)}_n',\quad
{\rm(c)}_{2n+1}'\Longrightarrow{\rm(a)}_n',
$$
$$
{\rm(a)}_n'\Longrightarrow{\rm(d)}_n'\ \mbox{if $n\ge2$},\quad
{\rm(c)}_{2n+1}'\Longrightarrow{\rm(d)}_n',\quad
{\rm(d)}_{2n+1}'\Longrightarrow{\rm(c)}_n'.
$$
\end{thm}

\begin{proof}
First, note that $\limsup_{t\searrow0}tf(t)\le0$ and $\lim_{t\to\infty}f(t)>-\infty$,
slightly stronger than the boundary conditions in (b)$_n'$--(d)$_n'$, are obvious as long
as $f$ satisfies (a)$_1'$, i.e., $f$ is non-decreasing on $(0,\infty)$.

(a)$_n'$ $\Rightarrow$ (b)$_n'$ if $n\ge2$.\enspace
Suppose (a)$_n'$ with $n\ge2$. The stated c.p.d.\ of $L_f$ is a consequence of L\"owner's
theorem. Next, we show that $\lim_{t\to\infty}f(t)/t\in[0,\infty)$, slightly stronger than
$\limsup_{t\to\infty}f(t)/t<+\infty$. By taking $f(t+1)-f(1)+1$ we may assume that
$f(t)>0$ for all $t>0$. Then it follows from (i)$_2$ $\Rightarrow$ (viii)$_1$ of Lemma
\ref{L-1.2} that $t/f(t)$ is non-decreasing on $(0,\infty)$, so the conclusion follows.

(b)$_{4n+1}'$ $\Rightarrow$ (a)$_n'$.\enspace
One can apply (b)$_{4n+1}$ $\Rightarrow$ (a)$_n$ of Theorem \ref{T-2.1} to $-f$ to see that
$f$ is $n$-concave on $(0,\infty)$. Thanks to $\limsup_{t\to\infty}f(t)>-\infty$ this
implies also that $f$ is non-decreasing on $(0,\infty)$. For any $\eps>0$ let
$f_\eps(t):=f(t+\eps)-f(\eps)+1$ for $t\ge0$, and apply (ii)$_n$ $\Rightarrow$ (i)$_n$ of
Lemma \ref{L-1.2} to $f_\eps$ so that $f_\eps$ is $n$-monotone on $[0,\infty)$. Hence $f$
is $n$-monotone on $(0,\infty)$ since $\eps>0$ is arbitrary.

(a)$_{2n+2}'$ $\Rightarrow$ (c)$_n'$.\enspace
It follows from (i)$_{2n+2}$ $\Rightarrow$ (ii)$_{n+1}$ of Lemma \ref{L-1.2} that $f$ is
$(n+1)$-concave on $(0,\infty)$. Now (c)$_n'$ is shown by applying (a)$_{n+1}$
$\Rightarrow$ (c)$_n$ of Theorem \ref{T-2.1} to $-f$.

(c)$_{2n+1}'$ $\Rightarrow$ (a)$_n'$ is proved similarly to (b)$_{4n+1}'$ $\Rightarrow$
(a)$_n'$ above. Indeed, apply (c)$_{2n+1}$ $\Rightarrow$ (a)$_n$ of Theorem \ref{T-2.1} to
$-f$ and use Lemma \ref{L-1.2} as above.

(a)$_n'$ $\Rightarrow$ (d)$_n'$ if $n\ge2$.
For any $\eps>0$, since $f(t+\eps)=(tf(t+\eps))/t$ is $n$-monotone on $(0,\infty)$, it
follows from (v)$_2$ $\Rightarrow$ (iii)$_1$ of Lemma \ref{L-1.1} that $tf(t+\eps)$ is
convex on $[0,\infty)$. Letting $\eps\searrow0$ yields that $tf(t)$ is convex on
$(0,\infty)$, from which we have $\liminf_{t\searrow0}t^2f(t)\ge0$, slightly stronger than
$\limsup_{t\searrow0}t^2f(t)\ge0$. For each $\eps>0$ let $g_\eps(t):=(t-\eps)^2f(t)$ for
$t\in(0,\infty)$. Note that the second divided difference $g_\eps^{[2]}(t,\eps,\eps)$ is
nothing but $f(t)$, which is $n$-monotone on $(0,\infty)$. Hence (a)$_n'$ implies by
\cite[p.\ 139, Lemma 5]{Do} that $L_{g_\eps}(t_1,\dots,t_n)$ is c.p.d.\ for all
$t_1,\dots,t_n\in(0,\infty)$. Letting $\eps\searrow0$ yields the stated c.p.d.\ of
$L_{t^2f(t)}$.

(c)$_{2n+1}'$ $\Rightarrow$ (d)$_n'$.\enspace
Although it is already known that (c)$_{2n+1}'$ $\Rightarrow$ (a)$_n'$ $\Rightarrow$
(d)$_n'$ if $n\ge2$, we here give an independent proof. Set $g(t):=tf(t)$ for $t>0$; then
(c)$_{2n+1}'$ implies that $g$ satisfies (b)$_{2n+1}$ of Theorem \ref{T-2.1}. Hence
Proposition \ref{P-2.5} implies that $g$ is convex on $(0,\infty)$ and so
$\liminf_{t\searrow0}t^2f(t)\ge0$. For each $\eps>0$ choose a constant
$\gamma_\eps<g'(\eps)$ and define
$$
g_\eps(t):=g(t+\eps)-g(\eps)-\gamma_\eps t,\qquad t\in[0,\infty).
$$
Note that $g_\eps(0)=0$, $g_\eps'(0)>0$, and $g_\eps(t)>0$ for all $t>0$. Since $g_\eps$
satisfies (vi)$_{2n+1}$, one can apply (vi)$_{2n+1}$ $\Rightarrow$ (ix)$_{2n}$
$\Rightarrow$ (v)$_n$ $\Rightarrow$ (vii)$_n$ of Lemmas \ref{L-1.4}, \ref{L-1.3}, and
\ref{L-1.5} to $g_\eps$ so that $L_{tg_\eps(t)}(t_1,\dots,t_n)$ is c.p.d.\ for all
$t_1,\dots,t_n\in(0,\infty)$. Then the asserted c.p.d.\ of $L_{tg(t)}=L_{t^2f(t)}$ is
shown in the same way as in the proof of (a)$_{n+1}$ $\Rightarrow$ (c)$_n$ of Theorem
\ref{T-2.1}.

(d)$_{2n+1}'$ $\Rightarrow$ (c)$_n'$.\enspace
Let $g$ be as above. Since (d)$_{2n+1}'$ implies that $g$ satisfies (c)$_{2n+1}$ of
Theorem \ref{T-2.1}, $g$ is convex on $(0,\infty)$ by Proposition \ref{P-2.5} (or by
(c)$_{2n+1}$ $\Rightarrow$ (a)$_n$ of Theorem \ref{T-2.1}), and so
$\lim_{t\to\infty}f(t)>-\infty$. For each $\eps>0$ define $g_\eps$ as above,
which satisfies (vii)$_{2n+1}$. Then one can apply (vii)$_{2n+1}$ $\Rightarrow$ (v)$_{2n}$
$\Rightarrow$ (ix)$_n$ $\Rightarrow$ (vi)$_n$ of Lemmas \ref{L-1.5}, \ref{L-1.3}, and
\ref{L-1.4} to $g_\eps$ so that $L_{g_\eps}(t_1,\dots,t_n)$ is c.n.d.\ for all
$t_1,\dots,t_n\in(0,\infty)$. This shows the asserted c.n.d.\ of $L_g=L_{tf(t)}$ by
letting $\eps\searrow0$.
\end{proof}

The equivalence of the following (a)$'$--(d)$'$ immediately follows from Theorem
\ref{T-2.6}. In \cite[Theorem 2.4]{Uc}, Uchiyama extended \cite[Theorem 2.4]{HP} in such
a way that a continuous function $f$ on $(0,\infty)$ is operator monotone if and only if
$f$ is operator concave and $\lim_{t\to\infty}f(t)>-\infty$ (or
$\limsup_{t\to\infty}f(t)>-\infty$). Due to this result, the equivalence of
(a)$'$, (b)$'$, and (c)$'$ is also an immediate consequence of Corollary \ref{C-2.2}.
Furthermore, the equivalence of (a)$'$, (c)$'$, and (d)$'$ extends
\cite[Theorems 1.1, 1.2, 1.4, and 1.5]{BS} as Corollary \ref{C-2.2} does. The equivalence
between (a)$'$ and (b)$'$ was proved in \cite[Theorem 3.3]{Uc} by a different method.

\begin{cor}\label{C-2.7}
Let $f$ be a real $C^1$ function on $(0,\infty)$. Then the following conditions are
equivalent:
\begin{itemize}
\item[\rm(a)$'$] $f$ is operator monotone on $(0,\infty)$;
\item[\rm(b)$'$] $\limsup_{t\to\infty}f(t)/t<+\infty$,
$\limsup_{t\to\infty}f(t)>-\infty$, and $L_f(t_1,\dots,t_n)$ is c.p.d.\ for all $n\in\bN$
and all $t_1,\dots,t_n\in(0,\infty)$;
\item[\rm(c)$'$] $\liminf_{t\searrow0}tf(t)\le0$,
$\limsup_{t\to\infty}f(t)>-\infty$, and $L_{tf(t)}(t_1,\dots,t_n)$ is c.n.d.\ for all
$n\in\bN$ and all $t_1,\dots,t_n\in(0,\infty)$;
\item[\rm(d)$'$] $\liminf_{t\searrow0}tf(t)\le0$,
$\limsup_{t\searrow0}t^2f(t)\ge0$, and $L_{t^2f(t)}(t_1,\allowbreak\dots,t_n)$ is
c.p.d.\ for all $n\in\bN$ and all $t_1,\dots,t_n\in(0,\infty)$.
\end{itemize}

Moreover, if the above conditions are satisfied, then $\lim_{t\to\infty}f(t)/t$,
$\lim_{t\to\infty}f(t)$, and $\lim_{t\searrow0}tf(t)$ exist in $[0,\infty)$,
$(-\infty,\infty]$, and $(-\infty,0]$, respectively, and
$\lim_{t\searrow0}t^\alpha f(t)=0$ for any $\alpha>1$.
\end{cor}

\begin{proof}
It remains to show the last assertion. Assume that $f$ is operator monotone on
$(0,\infty)$. The existence of $\lim_{t\to\infty}f(t)/t\in[0,\infty)$ and
$\lim_{t\to\infty}f(t)\in(-\infty,\infty]$ was seen in the proof of Theorem \ref{T-2.6}.
Since (c)$'$ implies that $-f$ satisfies (c) of Corollary \ref{C-2.2}, the existence of
$\lim_{t\searrow0}tf(t)\in(-\infty,0]$ follows from Corollary \ref{C-2.2}, so it is
obvious that $\lim_{t\searrow0}t^\alpha f(t)=0$ if $\alpha>1$.
\end{proof}

\begin{remark}\label{R-2.8}\rm
In the proof of (a)$_n'$ $\Rightarrow$ (d)$_n'$ (if $n\ge2$) of Theorem \ref{T-2.6} we
used a result from \cite[Chapter XV]{Do}. In this respect, the equivalence between
(a)$'$ and (d)$'$ has a strong connection to \cite[Theorem 10]{Ho} and
\cite[p.\ 139, Theorem III]{Do}, in which the following result was given: Let $g$ be a
$C^1$ function on an interval $(a,b)$ and $c$ any point in $(a,b)$. Then
$L_g(t_1,\dots,t_n)$ is c.p.d.\ for all $t_1,\dots,t_n\in(a,b)$ of any size $n$ if and
only if $f$ is of the form
$$
g(t)=g(c)+g'(c)(t-c)+(t-c)^2f(t)
$$
with an operator monotone function $f$ on $(a,b)$. This in particular says that a $C^1$
function $f$ on $(a,b)$ is operator monotone if and only if
$L_{(t-c)^2f(t)}(t_1,\dots,t_n)$ is c.p.d.\ for all $t_1,\dots,t_n\in(a,b)$, $n\in\bN$. An
essential difference between the last condition and (d)$'$ is that the point $c$ is inside
the domain of $f$ for the former while it is the boundary point $0$ of $(0,\infty)$ for
the latter. So it does not seem easy to prove (a)$'$ $\Leftrightarrow$ (d)$'$ based on the
above result in \cite{Ho,Do}.
\end{remark}

\begin{remark}\label{R-2.9}\rm
Consider operator monotone functions $f_\lambda(t):=t/(1-\lambda t)$ on $(-1,1)$ with
$\lambda\in(-1,1)$, so $tf_\lambda(t)$ is $g_\lambda$ in Remark \ref{R-2.3}. By considering
$f_\lambda|_{(0,1)}$ and $-f_\lambda|_{(0,1)}$ with $\lambda\in(0,1)$, we see that neither
(a)$'$ $\Rightarrow$ (c)$'$ nor (c)$'$ $\Rightarrow$ (a)$'$ of Corollary \ref{C-2.7} can
be extended to functions on a finite open interval $(0,b)$. Indeed, the right counterparts
of Theorem \ref{T-2.6} and Corollary \ref{C-2.7} for functions on a finite interval $(a,b)$
will be presented in Section 4 (see Theorem \ref{T-4.1} and Corollary \ref{C-4.2}).
\end{remark}

\begin{remark}\label{R-2.10}\rm
Any of boundary conditions as $t\searrow0$ or $t\to\infty$ in Theorem \ref{T-2.6} and
Corollary \ref{C-2.7} is essential. For instance, the functions $t^3$, $t^{-1}$, $-t$, and
$-t^{-2}$ on $(0,\infty)$ are not $2$-monotone; see Proposition \ref{P-3.1}\,(1) for $t^3$
and $-t^{-2}$, and $t^{-1}$ and $-t$ are even not increasing as a numerical function. By
taking account of Proposition \ref{P-3.1}, the functions $t^3$ and $-t$ show that (b)$'$
$\Rightarrow$ (a)$_2$ is not true without $\limsup_{t\to\infty}f(t)/t<+\infty$ and
$\limsup_{t\to\infty}f(t)>-\infty$, respectively. Similarly, consider the functions
$t^{-1}$ and $-t$ to see that the two boundary conditions of (c)$'$ are essential for
(c)$'$ $\Rightarrow$ (a)$_2$, and the functions $t^{-1}$ and $-t^{-2}$ for the two
boundary conditions of (d)$'$.
\end{remark}

\section{Examples: power functions}
\setcounter{equation}{0}

In this section we examine the conditions in Theorems \ref{T-2.1} and \ref{T-2.6} in the
cases of lower orders $n=2,3$ for the power functions $t^\alpha$ on $(0,\infty)$. In fact,
we sometimes used such examples of power functions in the preceding section, for instance,
in Remarks \ref{R-2.4} and \ref{R-2.10}. Elementary discussions on the c.p.d.\ and
c.n.d.\ properties of $t^\alpha$ based on the Cauchy matrix and the Schur product theorem
are found in \cite[Section 2]{BS2}.

\begin{prop}\label{P-3.1}
Consider the power functions $t^\alpha$ on $(0,\infty)$, where $\alpha\in\bR$. Then:
\begin{itemize}
\item[\rm(1)] $t^\alpha$ is $2$-monotone if and only if $0\le\alpha\le1$, or equivalently,
$t^\alpha$ is operator monotone. Moreover, $-t^\alpha$ is $2$-monotone if and only if
$-1\le\alpha\le0$.
\item[\rm(2)] $t^\alpha$ is $2$-convex if and only if either $-1\le\alpha\le0$ or
$1\le\alpha\le2$, or equivalently, $t^\alpha$ is operator convex.
\item[\rm(3)] $L_{t^\alpha}(t_1,t_2)$ is c.p.d.\ for all $t_1,t_2\in(0,\infty)$ if and
only if either $0\le\alpha\le1$ or $\alpha\ge2$.
\item[\rm(4)] $L_{t^\alpha}(t_1,t_2)$ is c.n.d.\ for all $t_1,t_2\in(0,\infty)$ if and
only if either $\alpha\le0$ or $1\le\alpha\le2$.
\item[\rm(5)] $L_{t^\alpha}(t_1,t_2,t_3)$ is c.p.d.\ for all $t_1,t_2,t_3\in(0,\infty)$ if
and only if either $0\le\alpha\le1$ or $2\le\alpha\le3$.
\item[\rm(6)] $L_{t^\alpha}(t_1,t_2,t_3)$ is c.n.d.\ for all $t_1,t_2,t_3\in(0,\infty)$ if
and only if either $-1\le\alpha\le0$ or $1\le\alpha\le2$.
\end{itemize}
\end{prop}

\begin{proof}
For (1) and (2) see \cite[Proposition 3.1]{HT}. Here note that $-t^\alpha$ is $2$-monotone
if and only if so is $t^{-\alpha}=-(-t^\alpha)^{-1}$. (3) and (4) are immediately seen
from \cite[p.\ 137, Lemma 3]{Do}. 

(5)\enspace
If $0\le\alpha\le1$ or $2\le\alpha\le3$, then $t^{\alpha-1}$ is operator convex on
$(0,\infty)$ and Corollary \ref{C-2.2} implies the c.p.d.\ condition here. For the
converse, one can easily check that a $3\times3$ real matrix
$\bmatrix a&d&e\\d&b&f\\e&f&c\endbmatrix$ is c.p.d.\ (resp., c.n.d.) if and only if
$$
a+c\ge2e\quad(\mbox{resp.},\ \ a+c\le2e),
$$
$$
b+c\ge2f\quad(\mbox{resp.},\ \ b+c\le2f),
$$
$$
(c+d-e-f)^2\le(a+c-2e)(b+c-2f).
$$
Since the above first two conditions for c.p.d.\ of $L_{t^\alpha}(t_1,t_2)$ means the
c.p.d.\ condition in (3), we must have $0\le\alpha\le1$ or $\alpha\ge2$ from (3). (This is
also obvious since the c.p.d.\ of order three implies that of order two.) The last
condition for c.p.d.\ of $L_{t^\alpha}(x,y,1)$ is written as
\begin{align}
&\biggl(\alpha+{x^\alpha-y^\alpha\over x-y}
-{x^\alpha-1\over x-1}-{y^\alpha-1\over y-1}\biggr)^2 \nonumber\\
&\qquad\le\biggl(\alpha(x^{\alpha-1}+1)-2\,{x^\alpha-1\over x-1}\biggr)
\biggl(\alpha(y^{\alpha-1}+1)-2\,{y^\alpha-1\over y-1}\biggr). \label{F-3.1}
\end{align}
Multiplying $(x-y)^2(x-1)^2(y-1)^2$ to the both sides of \eqref{F-3.1} gives
\begin{align}
&\bigl(\alpha(x-y)(x-1)(y-1)+x^\alpha(y-1)^2-(x-1)y^\alpha(y-1) \nonumber\\
&\quad+(x-y)(y-1)-(x-y)(x-1)(y^\alpha-1)\bigr)^2 \nonumber\\
&\qquad\le(x-y)^2(x-1)(y-1)F_\alpha(x)F_\alpha(y), \label{F-3.2}
\end{align}
where
$$
F_\alpha(x):=(\alpha-2)x^\alpha+\alpha x-\alpha x^{\alpha-1}-(\alpha-2).
$$
When $\alpha>2$, the left-hand side of \eqref{F-3.2} has the term $x^{2\alpha}$ of maximal
degree for $x$ with positive coefficient $(y-1)^4$, and the right-hand side has the term
$x^{\alpha+3}$ of maximal degree for $x$ with coefficient $(\alpha-2)(y-1)F_\alpha(y)$
which is positive for large $y>0$. Hence $2\alpha\le\alpha+3$ or $\alpha\le3$ is necessary
for \eqref{F-3.1} to hold for all $x,y>0$. So we must have $0\le\alpha\le1$ or
$2\le\alpha\le3$.

(6)\enspace
If $-1\le\alpha\le0$ or $1\le\alpha\le2$, then $t^\alpha$ is operator convex on
$(0,\infty)$ and Corollary \ref{C-2.2} implies the c.n.d.\ condition here. Conversely,
since the c.n.d.\ condition here implies that of order $2$ in (4), we must have
$\alpha\le0$ or $1\le\alpha\le2$ form (4). Moreover, \eqref{F-3.2} holds in this case too.
When $\alpha<0$, the left-hand side of \eqref{F-3.2} has the term $x^{2\alpha}$ of maximal
degree for $1/x$ with positive coefficient $(y-1)^4$, and the right-hand side has the term
$x^{\alpha-1}$ of maximal degree for $1/x$ with coefficient $\alpha y^2(y-1)F_\alpha(y)$
which is positive for small $y>0$. Hence $2\alpha\ge\alpha-1$ or $\alpha\ge-1$ must hold,
so we have $-1\le\alpha\le0$ or $1\le\alpha\le2$.
\end{proof}

Concerning the conditions of Theorem \ref{T-2.1} the above proposition shows that
(b)$_2$ $\Rightarrow$ (a)$_2$, (c)$_2$ $\Rightarrow$ (a)$_2$, (b)$_2$
$\Rightarrow$ (c)$_2$, and (c)$_2$ $\Rightarrow$ (b)$_2$ are all invalid while (a)$_2$,
(b)$_3$, and (c)$_3$ are equivalent for the power functions $t^\alpha$. Moreover,
concerning Theorem \ref{T-2.6}, we notice from the proposition that, restricted to the
power functions $t^\alpha$, conditions (a)$_2'$, (b)$_2'$, and (c)$_2'$ are equivalent
but (d)$_2'$ is strictly weaker.

\section{Functions on $(a,b)$}
\setcounter{equation}{0}

For a real $C^1$ function $f$ on $(a,\infty)$ where $-\infty<a<\infty$, we have the same
implications as in Theorem \ref{T-2.6} with slight modifications of (a)$_n'$--(d)$_n'$ by
applying the theorem to $f(t+a)$ on $(0,\infty)$. For example, (a)$_n'$ and (c)$_n'$ are
modified as
\begin{itemize}
\item[\rm(a)$_n'$] $f$ is $n$-monotone on $(a,\infty)$,
\item[\rm(c)$_n'$] $\liminf_{t\searrow a}(t-a)f(t)\le0$,
$\limsup_{t\to\infty}f(t)>-\infty$, and $L_{(t-a)f(t)}(t_1,\dots,t_n)$ is c.n.d.\ for
all $t_1,\dots,t_n\in(a,\infty)$,
\end{itemize}
and (b)$_n'$ and (d)$_n'$ are similarly modified.

Moreover, for a real $C^1$ function $f$ on $(-\infty,b)$ where $-\infty<b<\infty$, one can
apply Theorem \ref{T-2.6} to $-f(b-t)$ on $(0,\infty)$ so that the same implications as
there hold for the following conditions:
\begin{itemize}
\item[\rm(a)$_n''$] $f$ is $n$-monotone on $(-\infty,b)$;
\item[\rm(b)$_n''$] $\limsup_{t\to-\infty}f(t)/t<+\infty$,
$\liminf_{t\to-\infty}f(t)<+\infty$, and $L_f(t_1,\dots,t_n)$ is c.p.d.\ for all
$t_1,\dots,t_n\in(-\infty,b)$;
\item[\rm(c)$_n''$] $\limsup_{t\nearrow b}(b-t)f(t)\ge0$,
$\liminf_{t\to-\infty}f(t)<+\infty$, and $L_{(b-t)f(t)}(t_1,\dots,t_n)$ is c.n.d.\ for
all $t_1,\dots,t_n\in(-\infty,b)$;
\item[\rm(d)$_n''$] $\limsup_{t\nearrow b}(b-t)f(t)\ge0$,
$\liminf_{t\nearrow b}(b-t)^2f(t)\le0$, and $L_{(b-t)^2f(t)}(t_1,\dots,t_n)$ is
c.p.d.\ for all $t_1,\dots,t_n\in(-\infty,b)$.
\end{itemize}

The aim of this section is to prove the next theorem that is the counterpart of Theorem
\ref{T-2.6} for a real $C^1$ function on an finite open interval $(a,b)$.

\begin{thm}\label{T-4.1}
Let $f$ be a real $C^1$ function on $(a,b)$ where $-\infty<a<b<\infty$.
For each $n\in\bN$ consider the following conditions:
\begin{itemize}
\item[$(\alpha)_n$] $f$ is $n$-monotone on $(a,b)$;
\item[$(\beta)_n$] $\limsup_{t\nearrow b}(b-t)f(t)<+\infty$,
$\limsup_{t\nearrow b}f(t)>-\infty$, and $L_{(b-t)^2f(t)}(t_1,\dots,t_n)$ is c.p.d.\ for
all $t_1,\dots,t_n\in(a,b)$;
\item[$(\gamma)_n$] $\liminf_{t\searrow a}(t-a)f(t)\le0$,
$\limsup_{t\nearrow b}f(t)>-\infty$, and $L_{(t-a)(b-t)f(t)}(t_1,\dots,t_n)$
is c.n.d.\ for all $t_1,\dots,t_n\in(a,b)$;
\item[$(\delta)_n$] $\liminf_{t\searrow a}(t-a)f(t)\le0$,
$\limsup_{t\searrow a}(t-a)^2f(t)\ge0$, and $L_{(t-a)^2f(t)}(t_1,\dots,t_n)$
is c.p.d.\ for all $t_1,\dots,t_n\in(a,b)$.
\end{itemize}
Then for every $n\in\bN$ the following implications hold:
$$
(\alpha)_n\Longrightarrow(\beta)_n\ \mbox{if $n\ge2$},\quad
(\beta)_{4n+1}\Longrightarrow(\alpha)_n,\quad
(\alpha)_{2n+2}\Longrightarrow(\gamma)_n,\quad
(\gamma)_{2n+1}\Longrightarrow(\alpha)_n,
$$
$$
(\alpha)_n\Longrightarrow{\rm(\delta)}_n\ \mbox{if $n\ge2$},\quad
(\gamma)_{2n+1}\Longrightarrow(\delta)_n,\quad
(\delta)_{2n+1}\Longrightarrow(\gamma)_n.
$$
\end{thm}

\begin{proof}
Define a bijective function $\psi:(a,b)\to(0,\infty)$ by
$$
\psi(t):={t-a\over b-t}=-1+{b-a\over b-t},\qquad t\in(a,b),
$$
and hence
$$
\psi^{-1}(x)={bx+a\over x+1}=b-{b-a\over x+1},\qquad x\in(0,\infty).
$$
Furthermore, define a $C^1$ function $\tilde f$ on $(0,\infty)$ by
$\tilde f(x):=f(\psi^{-1}(x))$ for $x\in(0,\infty)$. The theorem immediately follows from
Theorem \ref{T-2.6} once we show that $(\alpha)_n$, $(\beta)_n$, $(\gamma)_n$, and
$(\delta)_n$ are equivalent, respectively, to (a)$_n'$, (b)$_n'$, (c)$_n'$, and (d)$_n'$
for $\tilde f$. First, the equivalence of $(\alpha)_n$ to (a)$_n'$ for $\tilde f$ is
immediate since both $\psi$ on $(a,b)$ and $\psi^{-1}$ on $(0,\infty)$ are operator
monotone. The following equalities are easy to check:
\allowdisplaybreaks{
\begin{align*}
\limsup_{x\to\infty}\tilde f(x)/x&={1\over b-a}\limsup_{t\nearrow b}(b-t)f(t), \\
\limsup_{x\to\infty}\tilde f(x)&=\limsup_{t\nearrow b}f(t), \\
\liminf_{x\searrow0}x\tilde f(x)&={1\over b-a}\liminf_{t\searrow a}(t-a)f(t), \\
\limsup_{x\searrow0}x^2\tilde f(x)&={1\over(b-a)^2}\limsup_{t\searrow a}(t-a)^2f(t).
\end{align*}
}Next, let $t_1,\dots,t_n\in(a,b)$ be arbitrary and let $x_i:=\psi(t_i)$ for $i=1,\dots,n$.
By direct computations we have
\begin{align*}
\tilde f^{[1]}(x_i,x_j)&={f(t_i)-f(t_j)\over\psi(t_i)-\psi(t_j)} \\
&={1\over b-a}(b-t_i)f^{[1]}(t_i,t_j)(b-t_j) \\
&={1\over b-a}\bigl\{((b-t)^2f(t))^{[1]}(t_i,t_j)+(b-t_i)f(t_i)+(b-t_j)f(t_j)\bigr\}, \\
\Bigl(x\tilde f(x)\Bigr)^{[1]}(x_i,x_j)
&={\psi(t_i)f(t_i)-\psi(t_j)f(t_j)\over\psi(t_i)-\psi(t_j)} \\
&={1\over b-a}(b-t_i)\biggl({t-a\over b-t}\,f(t)\biggr)^{[1]}(t_i,t_j)(b-t_j) \\
&={1\over b-a}\bigl\{\bigl((t-a)(b-t)f(t)\bigr)^{[1]}(t_i,t_j)
+(t_i-a)f(t_i)+(t_j-a)f(t_j)\bigr\}, \\
\Bigl(x^2\tilde f(x)\Bigr)^{[1]}(x_i,x_j)
&={\psi(t_i)^2f(t_i)-\psi(t_j)^2f(t_j)\over\psi(t_i)-\psi(t_j)} \\
&={1\over b-a}(b-t_i)\biggl(\biggl({t-a\over b-t}\biggr)^2f(t)\biggr)^{[1]}(t_i,t_j)
(b-t_j) \\
&={1\over b-a}\biggl\{\bigl((t-a)^2f(t)\bigr)^{[1]}(t_i,t_j)
+{(t_i-a)^2\over b-t_i}f(t_i)+{(t_j-a)^2\over b-t_j}f(t_j)\biggr\}.
\end{align*}
It is seen from the above equalities that $(\beta)_n$, $(\gamma)_n$, and $(\delta)_n$ are
equivalent, respectively, to (b)$_n'$, (c)$_n'$, and (d)$_n'$ for $\tilde f$.
\end{proof}

\begin{cor}\label{C-4.2}
Let $f$ be a real $C^1$ function on $(a,b)$ where $-\infty<a<b<\infty$. Then the following
conditions are equivalent:
\begin{itemize}
\item[$(\alpha)$] $f$ is operator monotone on $(a,b)$;
\item[$(\beta)$] $\limsup_{t\nearrow b}(b-t)f(t)<+\infty$,
$\limsup_{t\nearrow b}f(t)>-\infty$, and $L_{(b-t)^2f(t)}(t_1,\dots,t_n)$ is c.p.d.\ for
all $n\in\bN$ and all $t_1,\dots,t_n\in(a,b)$;
\item[$(\gamma)$] $\liminf_{t\searrow a}(t-a)f(t)\le0$,
$\limsup_{t\nearrow b}f(t)>-\infty$, and $L_{(t-a)(b-t)f(t)}(t_1,\dots,t_n)$ is c.n.d.\ for
all $n\in\bN$ and all $t_1,\dots,t_n\in(a,b)$;
\item[$(\delta)$] $\liminf_{t\searrow a}(t-a)f(t)\le0$,
$\limsup_{t\searrow a}(t-a)^2f(t)\ge0$, and $L_{(t-a)^2f(t)}(t_1,\dots,t_n)$ is c.p.d.\ for
all $n\in\bN$ and all $t_1,\dots,t_n\in(a,b)$.
\end{itemize}
\end{cor}

\begin{remark}\rm
Let $f$ be a $C^1$ function on a finite interval $(a,b)$ and $c$ be an arbitrary point in
$(a,b)$. As mentioned in Remark \ref{R-2.8}, it is known by \cite{Ho,Do} that $f$ is
operator monotone on $(a,b)$ if and only if $L_{(t-c)^2f(t)}(t_1,\dots,t_n)$ is c.p.d.\ for
all $t_1,\dots,t_n\in(a,b)$, $n\in\bN$. By letting $c\nearrow b$ and $c\searrow a$ it
follows that the $(\alpha)$ implies the c.p.d.\ conditions in $(\beta)$ and $(\delta)$.
Corollary \ref{C-4.2} says that the c.p.d.\ of $L_{(t-c)^2f(t)}$ for the boundary point
$c=b$ or $c=a$ with additional boundary conditions conversely implies the c.p.d.\ of
$L_{(t-c)^2f(t)}$ for all $c\in(a,b)$. On the other hand, it is known (see
\cite[Corollary 2.7.8]{Hi} and \cite[Lemma 2.1]{Uc}) that $f$ is operator convex on $(a,b)$
if and only if $f^{[1]}(c,\cdot)$ is opertor monotone on $(a,b)$ for some $c\in(a,b)$. So
one can also obtain characterizarions of the operator convexity of $f$ by applying
Corollary \ref{C-4.2} to $f^{[1]}(c,\cdot)$ when $f$ is assumed to be $C^2$ on $(a,b)$.
However, such characterizations are not so immediate to the function $f$ as those in
Corollary \ref{C-2.2} for $f$ on $(0,\infty)$.
\end{remark}

\begin{remark}\rm
Let $f_\lambda$, $\lambda\in[-1,1]$, be operator monotone functions on $(-1,1)$ given in
Remark \ref{R-2.9}, which are kernel functions in L\"owner's integral representation for
operator monotone functions on $(-1,1)$. Theorem \ref{T-4.1} says that
$L_{(1-t)^2f_\lambda(t)}(t_1,\dots,t_n)$ and $L_{(t+1)^2f_\lambda(t)}(t_1,\dots,t_n)$ are
c.p.d.\ and $L_{(1-t^2)f_\lambda}(t_1,\dots,t_n)$ is c.n.d.\ for every
$t_1,\dots,t_n\in(-1,1)$. Indeed, these can be directly checked by the following
expressions:
\begin{align*}
\Bigl((1-t)^2f_\lambda(t)\Bigr)^{[1]}(t_i,t_j)
&=-{t_i+t_j\over\lambda}+{2\lambda-1\over\lambda^2}
+{(\lambda^{-1}-1)^2\over(1-\lambda t_i)(1-\lambda t_j)}, \\
\Bigl((t+1)^2f_\lambda(t)\Bigr)^{[1]}(t_i,t_j)
&=-{t_i+t_j\over\lambda}-{2\lambda+1\over\lambda^2}
+{(\lambda^{-1}+1)^2\over(1-\lambda t_i)(1-\lambda t_j)}, \\
\Bigl((1-t^2)f_\lambda(t)\Bigr)^{[1]}(t_i,t_j)
&={t_i+t_j\over\lambda}+{1\over\lambda^2}
-{\lambda^{-2}-1\over(1-\lambda t_i)(1-\lambda t_j)}.
\end{align*}
Moreover, if $f$ is operator monotone on $(-1,1)$, then the boundary conditions as
$t\nearrow1$ or $t\searrow-1$ in $(\beta)$--$(\delta)$ are shown by L\"owner's integral
representation. Since an operator monotone functions on $(a,b)$ is transformed into that
on $(-1,1)$ by an affine function, the argument here supplies another (direct) proof of
the implications from $(\alpha)$ to $(\beta)$--$(\gamma)$ in Corollary \ref{C-4.2}. So the
converse implications of these are of actual substance in the corollary.
\end{remark}

\end{document}